\documentclass{daj}
\usepackage{amssymb}
\usepackage{amsmath}
\usepackage{amsthm}

\newcommand{\seq}{\subseteq}
\newcommand{\stm}{\setminus}
\newcommand{\est}{\varnothing}

\newcommand{\gam}{\gamma}
\renewcommand{\phi}{\varphi}

\newcommand{\R}{{\mathbb R}}
\newcommand{\Z}{{\mathbb Z}}

\newcommand{\tA}{{\tilde A}}

\newcommand{\cA}{{\mathcal A}}
\newcommand{\cL}{{\mathcal L}}

\newcommand{\prt}{\partial}

\newcommand{\<}{\langle}
\renewcommand{\>}{\rangle}

\newcommand{\longc}{,\ldots,}

\newcommand{\longp}{+\dots+}
\newcommand{\longu}{\cup\dots\cup}
\newcommand{\longo}{\oplus\dots\oplus}

\newcommand{\rk}{\mathop{\mathrm{rk}}}
\newcommand{\ord}{\mathop{\mathrm{ord}}}
\newcommand{\supp}{\mathop{\mathrm{supp}}}

\newcommand{\reft}[1]{\ref{t:#1}}
\newcommand{\refc}[1]{\ref{c:#1}}
\newcommand{\refm}[1]{\ref{m:#1}}
\newcommand{\refs}[1]{\ref{s:#1}}
\newcommand{\refx}[1]{\ref{x:#1}}

\newcommand{\refe}[1]{\eqref{e:#1}}

\newtheorem{claim}{Claim}
\newtheorem{theorem}{Theorem}
\newtheorem{corollary}{Corollary}
\newtheorem{primetheorem}{Theorem}

\theoremstyle{remark}
\newtheorem{example}{Example}

\dajAUTHORdetails{%
  title = {On Isoperimetric Stability}, 
  author = {Vsevolod F. Lev},
  plaintextauthor = {Vsevolod F. Lev},
  plaintexttitle = {On Isoperimetric Stability}, 
  runningtitle = {On Isoperimetric Stability},
  runningauthor = {Vsevolod F. Lev},
  copyrightauthor = {Vsevolod F. Lev},
  keywords = {stability, edge boundary},
}   

\dajEDITORdetails{%
   year={2018},
   number={14},
   received={19 September 2017},   
   published={6 August 2018},  
   doi={10.19086/da.3699},       
}   

\begin{document}

\begin{frontmatter}[classification=text]

\title{On Isoperimetric Stability} 

\author[lev]{Vsevolod F. Lev}

\begin{abstract}
We show that a non-empty subset of an abelian group with a small edge
boundary must be large; in particular, if $A$ and $S$ are finite, non-empty
subsets of an abelian group such that $S$ is independent, and the edge
boundary of $A$ with respect to $S$ does not exceed $(1-\gam)|S||A|$ with a
real $\gam\in(0,1]$, then $|A|\ge4^{(1-1/d)\gam |S|}$, where $d$ is the
smallest order of an element of $S$. Here the constant $4$ is best possible.

As a corollary, we derive an upper bound for the size of the largest
independent subset of the set of popular differences of a finite subset of an
abelian group. For groups of exponent $2$ and $3$, our bound translates into
a sharp estimate for the additive dimension of the popular difference set.

We also prove, as an auxiliary result, the following estimate of possible
independent interest: if $A\seq\Z^n$ is a finite, non-empty downset then,
denoting by $w(a)$ the number of non-zero components of the vector $a\in A$,
we have
  \[\frac1{|A|} \sum_{a\in A} w(a) \le \frac12\, \log_2 |A|.\]
\end{abstract}
\end{frontmatter}

\section{Summary of Results}\label{s:sumres}

Let $G$ be an abelian group. For finite subsets $A,S\seq G$, denote by
$\prt_S(A)$ the number of edges from $A$ to its complement $G\setminus A$ in
the directed Cayley graph, induced on $G$ by $S$; that is,
  $$ \prt_S(A) = | \{ (a,s) \in A\times S \colon a+s \notin A \} |. $$
In a trivial way, we have $0\le \prt_S(A)\le |S||A|$. Clearly, if
$\prt_S(A)=0$, then $A$ is a union of cosets of the subgroup $\<S\>$
generated by $S$; in particular, $|A|\ge|\<S\>|$. What can be said about $A$
if we are given that $\prt_S(A)\le(1-\gam)|S||A|$, with some $\gam\in(0,1]$? How small
can $|A|$ be under this assumption? In the case where $G$ is homocyclic of
exponent not exceeding $4$ (that is, $G=C_m^n$ with $m\in\{2,3,4\}$ and $n\ge
1$), and $S$ is a standard generating subset of $G$, the answer is given by
the following easy consequence of \cite[Corollary~1.10]{b:l}.
\begin{theorem}\label{t:exp234}
Let $G$ be a homocyclic group of exponent $\exp(G)\in\{2,3,4\}$ and rank
$n:=\rk G$. If $A\seq G$ is non-empty and $\prt_S(A)\le(1-\gam)n|A|$ with a
generating subset $S\seq G$ and real $\gam\in(0,1]$, then
  $$ |A| \ge |G|^\gam. $$
\end{theorem}

We postpone the proofs of Theorem~\reft{exp234} and other results presented
in this section to Sections~\refs{avweightZ} -- \refs{repa}.

The following examples show that the assumption of Theorem~\reft{exp234}
cannot be relaxed to $\prt_S(A)\le(1-\gam)|S||A|$, and that the conclusion
$|A|\ge|G|^\gam$ is nearly best possible.
\begin{example}
Fix $m\ge 2$ and suppose that $k,n\ge 1$ are integers such that
 $k=\log_m n+O(1)$, with an absolute implicit constant. Let
$\{e_1\longc e_n\}$ be a standard generating subset of the group $C_m^n$, and
consider the sets $A:=\<e_1\longc e_k\>$ and
 $S:=A\cup\{e_{k+1}\longc e_n\}$. (We keep using the standard notation $\<T\>$ for
the subgroup of an abelian group generated by its subset $T$.) We have then
$|S|=m^k+n-k$ and $\prt_S(A)=(n-k)|A|=(1-\gam)|S||A|$ where, writing
$\tau:=m^k/n$,
  $$ \gam = \frac{m^k}{m^k+n-k} \approx \frac{\tau}{\tau+1}. $$
At the same time, $|A|=m^k$ is much smaller than $|C_m^n|^\gam=m^{\gam n}$.
\end{example}

\begin{example}\label{x:1}
Suppose that $m\ge 2$ and $k,n\ge 1$ are integers with $k\mid n$. Let
$C_m^n=H_1\longo H_k$ be a direct sum decomposition with every subgroup $H_i$
isomorphic to $C_m^{n/k}$, and let $S\seq C_m^n$ be an $n$-element generating
subset having exactly $n/k$ elements in every subgroup $H_i$. Consider the
set $A:=H_1\longu H_k$. It is easily seen that $|A|=(m^{n/k}-1)k+1$ and
$\prt_S(A)=(m^{n/k}-1)(k-1)n$. Consequently, letting $\gam:=k^{-1}$, we have
$\prt_S(A)<(1-\gam)n|A|$, while $|A|\le m^{n/k}k=\gam^{-1}|C_m^n|^{\gam}$.
\end{example}

In the case where $\exp(G)\in\{2,3\}$, Theorem~\reft{exp234} is easy to
extend to the situation where $S$ is an arbitrary (not necessarily
generating) subset of $G$.
\begin{corollary}\label{c:cosetdecomp}
Suppose that $A$ and $S$ are finite, non-empty subsets of an abelian group
$G$ of exponent $\exp(G)\in\{2,3\}$. Let $H:=\<S\>$ and write $n:=\rk H$. If
$\prt_S(A)\le(1-\gam)n|A|$ with a real $\gam\in(0,1]$, then
  $$ |A| \ge |H|^\gam. $$
\end{corollary}

Interestingly, Theorem~\reft{exp234} does not allow a straightforward
extension to the case where $\exp(G)>4$; this is demonstrated by our next
example.
\begin{example}
Suppose that $1<t<m$ and $n\ge 1$ are integers, and consider the set
$A:=[0,t-1]^n\seq C_m^n$; thus, $|A|=t^n$, and it is easily seen that if $S$
is a standard generating set in $C_m^n$, then $\prt_S(A)=nt^{n-1}$.
Consequently, letting $\gam:=1-t^{-1}$, we have $\prt_S(A)=(1-\gam)n|A|$
while $|A|=b^{\gam n}$, where $b=t^{\gam^{-1}}=\exp(t\log t/(t-1))$ can be as
small as $4$ (attained for $t=2$). Therefore, instead of the estimate
$|A|\ge|G|^\gam$ of Theorem~\reft{exp234}, the best estimate one can hope for
in the general case is $|A|\ge 4^{\gam n}$, where $n=\rk G$.
\end{example}

We say that a finite subset $S$ of an abelian group is \emph{independent} if
for any integer-valued function $k$ on $S$, we have
 $\sum_{s\in S} k(s)s\ne 0$, unless all summands are equal to $0$; that is, the
sum $\oplus_{s\in S}\<s\>$ is direct.

Extending Theorem~\reft{exp234} to arbitrary abelian groups, we have the
following result.
\begin{theorem}\label{t:generalcase}
Suppose that $A$ and $S$ are finite, non-empty subsets of an abelian group
such that $S$ is independent. Write $n:=|S|$ and $d:=\min\{\ord s\colon s\in
S\}$. If $\prt_S(A)\le(1-\gam)n|A|$ with a real $\gam\in(0,1]$, then
  $$ |A| \ge 4^{(1-1/d)\gam n}. $$
\end{theorem}

The statement of Theorem~\reft{generalcase} is to be interpreted the expected
way if some, or all, of the orders of the elements of $S$ are infinite; in
particular, in the latter case the conclusion of the theorem should be read
as $|A|\ge4^{\gam n}$.

As Example~\refx{1} shows, the coefficient $1-1/d$ in the exponent is best
possible for $d=2$, and cannot be replaced with a number larger than
$\log3/\log4\approx0.792$ for $d=3$.

Let $\Z_{\ge0}$ denote the set of non-negative integers. A set
$A\seq\Z^n_{\ge 0}$ is called a \emph{downset} if for every $a\in A$ and
every $z\in\Z^n_{\ge 0}$ majorated by $a$ coordinate-wise, we have $z\in A$.
The \emph{weight} of a vector $z\in\Z^n$, denoted below $w(z)$, is the number
of non-zero coordinates of $z$.

The following estimate for the average weight of a vector in a downset in
$\Z_{\ge 0}^n$ is an important ingredient of the proof of
Theorem~\reft{generalcase} and, we believe, may be found interesting in its own
right.

\begin{theorem}\label{t:avweightZ}
If $n\ge 1$ is an integer and $A\seq\Z_{\ge 0}^n$ is a finite, non-empty
downset, then
  $$ \frac1{|A|} \sum_{a\in A} w(a) \le \frac12\, \log_2 |A|. $$
\end{theorem}

It is easy to see that for sets of the form
$A=[0,l_1]\times\dots\times[0,l_n]$ with $l_1\longc l_n\in\{0,1\}$, equality
is attained in the estimate of Theorem~\reft{avweightZ}.

Theorem~\reft{avweightZ} looks strikingly similar to \cite[Theorem 1.1]{b:r}
which says that if a set $A\seq\{0,1\}^n$ is union-closed (that is,
$a_1,a_2\in A$ implies $a_1\vee a_2\in A$), then $\frac1{|A|} \sum_{a\in
A}w(a)\ge\frac12\log_2|A|$. However, it seems that the two results cannot be
reduced to each other.

For a somewhat unexpected link between (the proof of)
Theorem~\reft{avweightZ} and the Loomis-Whitney inequality, see a remark in
Section~\refs{concrem}.

In Section~3 we extend Theorem~\reft{avweightZ} to arbitrary abelian
groups.

One can equivalently restate Theorem~\reft{avweightZ} in terms of multisets.
Formally, a (finite) multiset with the ground set $S$ is a finitely supported
function from $S$ to $\Z_{\ge 0}$; the value of the function at the element
$s\in S$ is the \emph{multiplicity} of $s$ in the multiset. The cardinality
of the multiset $A$ is the sum $\sum_{s\in\supp A} A(s)$. For a multiset $A$
and an element $s\in S$, we write $s\in A$ if $A(s)>0$, and we define the
multiset $A\stm\{s\}$ by
  \[ A\stm\{s\}\colon t\mapsto
      \begin{cases}
        A(t), & \quad t\in S,\ t\ne s, \\
        \max\{A(s)-1,0\}, & \quad t=s.
      \end{cases} \]
The family $\cA$ of multisets (sharing the same ground set $S$) is
\emph{monotonic} if $A\in\cA$ implies $A\stm\{s\}\in\cA$ for each $s\in S$.

Since every element of $\Z^n_{\ge 0}$ defines a multiset with the ground
set $[n]$, and a set in $\Z^n_{\ge 0}$ is a downset if and only if the
corresponding family of multisets is monotonic, we have the following
restatement of Theorem~\reft{avweightZ}.
\begin{primetheorem}
If $\cA$ is a finite, non-empty, monotonic family of multisets, then
  $$ \frac1{|\cA|} \sum_{A\in\cA} |\supp A| \le \frac12\,\log_2|\cA|. $$
\end{primetheorem}

For a finite subset $A$ of an abelian group $G$, denote by $\dim_I(A)$ the
largest size of an independent subset of $A$. Given an element $g\in G$,
define $r_A(g)$ to be the number of representations of $g$ as a difference of
two elements of $A$, and for real $\gam\in(0,1]$ let $P_\gam(A)$ be the set
of all $\gam$-popular differences in $A$; that is,
  $$ r_A(g):=|\{(a,a')\in A\times A\colon g=a-a'\}| $$
and
  $$ P_\gam(A):=\{ g\in G\colon r_A(g)\ge\gam|A|\}. $$

As a simple corollary of Theorems~\reft{exp234} and~\reft{generalcase}, in
Section~\refs{repa} we prove
\begin{theorem}\label{t:repa}
If $p$ is the smallest order of a non-zero element of an abelian group $G$,
then for any finite, non-empty subset $A\seq G$ and real $\gam\in[0,1)$, we
have
  $$ \dim_I(P_\gam(A)) \le (2(1-1/p))^{-1}\gam^{-1}\log_2|A|. $$
Moreover, if $\exp(G)=3$, then indeed we have
  $$ \dim_I(P_\gam(A)) \le \gam^{-1}\log_3|A|. $$
\end{theorem}

In the situation where $G$ is homocyclic of exponent $m$, the estimate of
Theorem~\reft{repa} is sharp for $m\in\{2,3\}$, and reasonably close to sharp
for $m\ge 4$. To see this we essentially return back to Example~\refx{1} to
review it from a slightly different perspective.
\begin{example}
Fix integers $m\ge 2$ and $k,n\ge 1$ with $k\mid n$, and consider a direct
sum decomposition $C_m^n=H_1\oplus\dotsb\oplus H_k$ where each of $H_1\longc
H_k<C_m^n$ is isomorphic to $C_m^{n/k}$. Let $A:=H_1\cup\dotsb\cup H_k$, so
that $|A|=k(m^{n/k}-1)+1\le km^{n/k}$ and every non-zero element $a\in A$
satisfies $r_A(a)=m^{n/k}$. Setting $\gam:=m^{n/k}/|A|\ge k^{-1}$, we then
have
  $$ \dim_I(P_\gam(A)) \ge n =
      k\log_m(\gam|A|)\ge\gam^{-1}\log_m|A|-\gam^{-1}\log_m(\gam^{-1}). $$
\end{example}

It is interesting to compare Theorem~\reft{repa} with a result of Shkredov
and Yekhanin \cite[Theorem~3.1]{b:sy}. To this end we recall that a subset
$A$ of an abelian group is called \emph{dissociated} if the subset sums
$\sum_{a\in B}a$ are pairwise distinct, for all subsets $B\seq A$. The
\emph{additive dimension} of $A$, which we denote $\dim_D(A)$, is the size of
the largest dissociated subset of $A$. The result of Shkredov-Yekhanin
essentially says that if $A$ is a subset of a finite abelian group, then
\begin{equation}\label{e:sy}
  \dim_D(P_\gam(A)) \ll \gam^{-1}\log |A|
\end{equation}
with an absolute implicit constant. It is readily seen that every independent
set in an abelian group is dissociated, and that for the groups of exponent
$2$ and $3$, the two notions coincide. As a result, we have
$\dim_I(P)\le\dim_D(P)$, for every subset $P$ of the group, with equality for
groups of exponent $2$ or $3$. Consequently, the Shkredov-Yekhanin bound
\refe{sy} is qualitatively stronger than Theorem~\reft{repa} for groups of
exponent larger than $3$, while Theorem~\reft{repa} is stronger than
\refe{sy} for groups of exponent $2$ and $3$ (providing the sharp
coefficients in this case).

We now turn to the proofs of the results discussed above;
Theorem~\reft{avweightZ} will be proved in the next section,
Theorems~\reft{exp234} and~\reft{generalcase} and
Corollary~\refc{cosetdecomp} in Section~\refs{isoper}, and
Theorem~\reft{repa} in Section~\refs{repa}. Concluding remarks and open
problems are gathered in Section~\refs{concrem}.

\section{Proof of Theorem~\reft{avweightZ}}\label{s:avweightZ}

Given an integer $n\ge 1$, let $\{e_1\longc e_n\}$ be the standard basis of
$\R^n$, and for each $i\in[1,n]$ denote by $\cL_i$ the $i$th coordinate
hyperplane, and by $\pi_i$ the orthogonal projection of $\R^n$ onto $\cL_i$.

We use induction by $n$, and for every fixed value of $n$ by $|A|$. If $n=1$,
then
  $$ \frac{1}{|A|} \sum_{a\in A}w(a) = 1-\frac1{|A|}
                                               \le \frac12 \log_2|A| $$
as one can easily verify. Also,  the estimate in question is immediate if
$|A|=1$. Suppose therefore that $n\ge 2$ and also $|A|\ge 2$.

Since the set $A$ is a downset, it has exactly $|\pi_i(A)|$ elements on the
$i$th coordinate hyperplane $\cL_i$, for every $i\in[1,n]$. Consequently,
double-counting gives
\begin{align}
  \sum_{a\in A} w(a)
     &= \sum_{i=1}^n | \{ a\in A\colon e_i\in\supp a \} | \notag \\
     &= \sum_{i=1}^n (|A|-|\pi_i(A)|) \notag \\
     &= n|A| - (|\pi_1(A)|\longp|\pi_n(A)|), \notag
\end{align}
and we thus want to prove that
\begin{equation*}
  n|A| \le |\pi_1(A)|\longp|\pi_n(A)| + \frac12 |A| \log_2|A|.
\end{equation*}

Using again the assumption that $A$ is a downset, we conclude that its
projection onto the $n$th coordinate axis is an interval $[0,l]$ with an
integer $l\ge 0$, and we partition $A$ as $A=B\cup(le_n+C)$ where
$C\seq\cL_n$ and $B\cap(le_n+\cL_n)=\est$.

If $B=\est$, then $A\seq\cL_n$; consequently,
  $$ (n-1)|A| \le |\pi_1(A)|\longp|\pi_{n-1}(A)| + \frac12 |A| \log_2|A| $$
by the induction hypothesis, and combining this with $|\pi_n(A)|=|A|$ we
get the assertion.

If $B\ne\est$, then both $B\seq\Z_{\ge0}^n$ and $C\seq\Z_{\ge0}^{n-1}$ are
downsets, and the induction hypothesis gives
\begin{equation}\label{e:phi1}
  (n-1)|C| \le |\pi_1(C)|\longp|\pi_{n-1}(C)| + \frac12 |C| \log_2|C|
\end{equation}
and
\begin{equation}\label{e:phi2}
  n|B| \le |\pi_1(B)|\longp|\pi_n(B)| + \frac12 |B| \log_2|B|.
\end{equation}
Since $|\pi_i(A)|=|\pi_i(B)|+|\pi_i(C)|$ for $i\in[1,n-1]$, and
$\pi_n(A)=\pi_n(B)$ by the downset assumption, to complete the proof it
suffices to show that
\begin{equation}\label{e:phi3}
  |C| + \frac12 |C| \log_2|C| + \frac12 |B| \log_2|B|
                                                \le \frac12|A|\log_2|A|;
\end{equation}
the assertion will then follow by adding together \refe{phi1},
\refe{phi2}, and \refe{phi3}. To prove \refe{phi3} we let
$\tau:=|B|/|C|$, so that $\tau\ge 1$. Dividing through by $|C|$ and
substituting $|B|=\tau|C|$ and $|A|=(\tau+1)|C|$ into \refe{phi3} brings
it to the form
  $$ 1 + \frac12 \tau\log_2\tau \le \frac12(\tau+1)\log_2(\tau+1), $$
an inequality which is easy to verify using basic calculus.

\section{Proofs of Corollary~\refc{cosetdecomp} and Theorems~\reft{exp234}
         and~\reft{generalcase}}\label{s:isoper}

As indicated above, Theorem~\reft{exp234} is an immediate consequence of
\cite[Corollary~1.10]{b:l}, which says that if $G$ is a finite abelian group of
exponent $m\in\{2,3,4\}$, then for any generating subset
 $S\seq G$ and any non-empty subset $A\seq G$ one has
$\prt_S(A)\ge |A|\log_m(|G|/|A|)$; combined with the assumption
$\prt_S(A)\le(1-\gam)n|A|$, this yields $\log_m(|G|/|A|)\le(1-\gam)n$, and
the assertion of Theorem~\reft{exp234} follows.

\begin{proof}[Proof of Corollary~\refc{cosetdecomp}]
Consider the coset decomposition $A=(a_1+A_1)\cup\cdots\cup(a_k+A_k)$, where
$A_1\longc A_k\seq H$, and $a_1\longc a_k$ represent pairwise distinct
$H$-cosets. The case where $S=\{0\}$ is immediate, and we thus assume that
$S\ne\{0\}$, and therefore $H$ is non-trivial. From
  $$ \sum_{i=1}^k \prt_S(A_i) = \prt_S(A) \le (1-\gam)n|A|
                                            = (1-\gam)n\sum_{i=1}^k |A_i| $$
we conclude that there exists $i\in[1,k]$ with
$\prt_S(A_i)\le(1-\gam)n|A_i|$. Since $\exp(G)\in\{2,3\}$ implies that also
$\exp(H)\in\{2,3\}$, and therefore $H$ is homocyclic, by
Theorem~\reft{exp234} we now get
  $$ |A| \ge |A_i| \ge |H|^\gam, $$
as required.
\end{proof}

Below in this section we prove Theorem~\reft{generalcase}. The proof uses the
compression technique in the general settings of arbitrary finite abelian
groups, so we start with setting up the machinery of this general
compression.

Suppose that $G$ is an abelian group, and
 $S=\{s_1\longc s_n\}\seq G$ is a finite, independent, generating subset; that is,
for each $i\in[1,n]$, letting $S_i:=S\stm\{s_i\}$, we have
$G=\<S_i\>\oplus\<s_i\>$. For integer $k\ge 0$ and elements $g,v\in G$ with
$\ord v\ge k$, write $P(g,v,k):=\{g,g+v\longc g+(k-1)v\}$; thus, $P(g,v,k)$
is the $k$-term arithmetic progression with the first term $g$ and difference
$v$. Given a finite subset $A\seq G$ and an index $i\in[1,n]$, define the
compression of $A$ along $s_i$ to be the set
  $$ [A]_i := \bigcup_{g\in\<S_i\>} P\big(g,s_i,|A\cap(g+\<s_i\>)|\big); $$
that is, in each $\<s_i\>$-coset the set $[A]_i$ has exactly as many elements
as the original set $A$, but the elements are ``stacked towards the beginning
of the coset''. Thus, writing for every $i\in[1,n]$
  $$ K_i := \begin{cases}
              [0,\ord(s_i))\ &\text{if $s_i$ has finite order,} \\
              \Z\ &\text{otherwise},
            \end{cases} $$
for an element $g\in G$ to lie it $[A]_i$, it is necessary and sufficient
that in the (unique) representation $g=h+ks_i$ with $h\in\<S_i\>$ and $k\in
K_i$, we had in fact $0\le k<|(g+\<s_i\>)\cap A|$.

We say that $A$ is $i$-compressed if $[A]_i=A$; that is, for each
$g\in\<S_i\>$ and $k\in K_i$, we have $g+ks_i\in A$ if and only if $0\le
k<|(g+\<s_i\>)\cap A|$. Equivalently, $A$ is $i$-compressed if
$A\stm\<S_i\>\seq A+s_i$. Clearly, $[A]_i$ is $i$-compressed for any
 $A\seq G$ and $i\in[1,n]$.

The set $A\seq G$ is \emph{compressed with respect to $S$} if it is
$i$-compressed for each $i\in[1,n]$.

We now show that compression along any element of $S$ does not destroy the
property of being compressed along other elements of $S$.
\begin{claim}\label{m:comp1}
If, for some $i,j\in[1,n]$, the set $A\seq G$ is $i$-compressed, then so
is the set $[A]_j$.
\end{claim}

\begin{proof}
The assertion is immediate if $i=j$, and we thus assume that $i\ne j$. We
further assume that $A$ is $i$-compressed and show that then also $[A]_j$ is
$i$-compressed.

It suffices to show that for any $g\in[A]_j\stm\<S_i\>$, we have
$g\in[A]_j+s_i$. Since $A$ is $i$-compressed, for any $z\in\<s_j\>$ with
$g+z\in A$ we also have $g-s_i+z\in A$; consequently,
  $$ |(g+\<s_j\>)\cap A| \le |(g-s_i+\<s_j\>)\cap A|, $$
whence
\begin{equation}\label{e:comp1}
  |(g+\<s_j\>)\cap[A]_j| \le |(g-s_i+\<s_j\>)\cap[A]_j|.
\end{equation}
Write $g=h+ks_j$ with $h\in\<S_j\>$ and $k\in K_j$. From $g\in[A]_j$ we have
$0\le k<|(g+\<s_j\>)\cap[A]_j|$. Now \refe{comp1} yields $0\le
k<|(g-s_i+\<s_j\>)\cap[A]_j|=|(h-s_i+\<s_j\>)\cap[A]_j|$, and it follows that
$g-s_i=(h-s_i)+ks_j\in[A]_j$, as wanted.
\end{proof}

\begin{claim}\label{m:comp2}
For any $G$, $A$, and $S$ as above, and any $i,j\in[1,n]$, we have
  $$ | \{ a\in [A]_i\colon a+s_j\notin [A]_i \} |
                               \le | \{ a\in A\colon a+s_j\notin A \} |. $$
Consequently,
  $$ \prt_S([A]_i) \le \prt_S(A). $$
\end{claim}

\begin{proof}
The assertion follows by decomposing
  $$ | \{ a\in A\colon a+s_j\notin A \} |
       = \sum_{g\in\<S_i\>} | \{ a\in (g+\<s_i\>)\cap A
                                               \colon a+s_j\notin A \} | $$
and
  $$ | \{ a\in [A]_i\colon a+s_j\notin [A]_i \} |
      = \sum_{g\in\<S_i\>} | \{ a\in (g+\<s_i\>)\cap [A]_i
                                          \colon a+s_j\notin [A]_i \} |, $$
and observing that, for any fixed $g\in\<S_i\>$, one has
  $$ | \{ a\in (g+\<s_i\>)\cap A\colon a+s_j\notin A \} |
         \ge \max \{ |(g+\<s_i\>)\cap A| - |(g+s_j+\<s_i\>)\cap A|, 0 \}, $$
and that equality holds if $A$ gets replaced with $[A]_i$:
\begin{multline*}
  | \{ a\in (g+\<s_i\>)\cap [A]_i\colon a+s_j\notin [A]_i \} |
           = \max \{|(g+\<s_i\>)\cap [A]_i|-|(g+s_j+\<s_i\>)\cap [A]_i|,0\}
\end{multline*}
(recall that both intersections $(g+\<s_i\>)\cap [A]_i$ and
$(g+s_j+\<s_i\>)\cap [A]_i$ are initial segments of the corresponding
cosets $g+\<s_i\>$ and $g+s_j+\<s_i\>$.)
\end{proof}

Claims \refm{comp1} and \refm{comp2} show that if, starting with $A$, one
subsequently applies compressions along each of the elements $s_1\longc s_n$,
then the resulting set $\tA$ is compressed with respect to $S$, and we have
$|\tA|=|A|$ and $\prt_S(\tA)\le\prt_S(A)$.

We need the following corollary (indeed, a generalization) of
Theorem~\reft{avweightZ}.
\begin{corollary}\label{c:avweightG}
Suppose that $S$ is a finite, independent generating set in an abelian group
$G$. If $A\seq G$ is finite, non-empty, and compressed with respect to $S$,
then
  $$ \frac1{|A|} \sum_{a\in A} w(a) \le \frac12\, \log_2 |A|, $$
where $w(a)$ is the number of non-zero summands in the representation of $a$
as a linear combination of the elements from $S$.
\end{corollary}

\begin{proof}
Let $n:=|S|$, write $S=\{s_1\longc s_n\}$, and consider the injective mapping
$\phi\colon G\to\Z^n$ defined by
  $$ \phi(z_1s_1\longp z_ns_n)
                            = (z_1\longc z_n),\ z_i\in K_i\ (1\le i\le n). $$
Clearly, the weight function on $G$ (with respect to $S$) agrees with that on
$\phi(G)$ (with respect to the standard basis of $\R^n$), and the assumption
that $A$ is compressed with respect to $S$ ensures that the image $\phi(A)$
is a downset. The assertion now follows by applying Theorem~\reft{avweightZ}
to $\phi(A)$.
\end{proof}

\begin{proof}[Proof of Theorem~\reft{generalcase}]
Let $G$ denote the underlying abelian group. We can assume that $S$ generates
$G$; once the assertion is established in this case, the general case follows
easily by considering the coset decomposition of $A$ as in
Corollary~\refc{cosetdecomp}.

In view of the remark following Claim~\refm{comp2}, we can also assume
without loss of generality that $A$ is compressed with respect to $S$.

Finally, we assume that at least one element of $S$ has finite order; the
modifications to be made in the argument below in the case where all elements
of $S$ have infinite order are straightforward (and indeed, the proof gets
only simpler in this case).

Write $S=\{s_1\longc s_n\}$, and for every $i\in[1,n]$, let
  $$ \prt_S^{(i)}(A) := | \{ a \in A \colon a+s_i \notin A \} |. $$
Define $N_i$ to be the number of $\<s_i\>$-cosets having a non-empty
intersection with $A$, and let $N_i'$ be the number of $\<s_i\>$-cosets
entirely contained in $A$; thus, $N_i'=0$ if the order of $s_i$ is infinite.
We have
\begin{equation}\label{e:decompsum}
   \prt_S(A)=\prt_S^{(1)}(A)\longp\prt_S^{(n)}(A)
\end{equation}
and, since $A$ is compressed,
\begin{equation}\label{e:generalcaseproof1}
  \prt_S^{(i)}(A) = N_i-N_i',\quad i\in[1,n].
\end{equation}
If $s_i$ is of finite order then, writing $d_i:=\ord s_i$, in a trivial way
we have
\begin{equation}\label{e:generalcaseproof2}
  \prt_S^{(i)}(A) \le |A| - d_i N_i' \le |A| - d N_i',
\end{equation}
and the resulting estimate remains true if $s_i$ is of infinite order. From
\refe{generalcaseproof1} and \refe{generalcaseproof2},
  $$ (d-1)\prt_S^{(i)}(A) \ge d N_i - |A|, $$
and then \refe{decompsum} along with the assumption
$\prt_S(A)\le(1-\gam)n|A|$ yield
\begin{equation}\label{e:generalcaseproof3}
   (d-1)(1-\gam)n|A| \ge (d-1) \prt_S(A) \ge d \sum_{i=1}^n N_i - n|A|.
\end{equation}
To address the sum in the right-hand side we notice that $N_i=|\<S_i\>\cap
A|=|A|-|A\stm\<S_i\>|$ for each $i\in[1,n]$ by the compression assumption.
Consequently, in view of
  $$ w(g) := | \{ i\in[1,n]\colon g\notin\<S_i\> \} |, \quad g\in G, $$
we have
  $$ \sum_{i=1}^n N_i = n|A| - \sum_{a\in A} w(a). $$
Substituting into \refe{generalcaseproof3} and simplifying we get
  $$ \frac1{|A|} \sum_{a\in A} w(a) \ge (1-1/d)\gam n, $$
and the result is now immediate from Corollary~\refc{avweightG}.
\end{proof}

\section{Proof of Theorem~\reft{repa}}\label{s:repa}

Let $S$ be an independent subset of the popular difference set $P_\gam(A)$,
and write $n:=|S|$. Every element $s\in S$ has at least $\gam |A|$
representations as $s=a'-a$ with $a,a'\in A$. This results in at least
$\gam|A|\cdot|S|$ pairs $(a,s)\in A\times S$ with $a+s\in A$. Hence,
  $$ \prt_S(A) \le |A||S| - \gam|A||S| = (1-\gam) n|A|. $$
Consequently,
  $$ |A| \ge 4^{(1-1/p)\gam n} $$
by Theorem~\reft{generalcase}, which yields
$n\le(2(1-1/p))^{-1}\gam^{-1}\log_2|A|$. The first estimate follows by
choosing $S$ so as to have $n=\dim_I(P_\gam(A))$.

The second estimate is obtained by replacing Theorem~\reft{generalcase} with
Corollary~\refc{cosetdecomp} to get
  $$ |A| \ge |G|^\gam \ge 3^{\gam n}. $$

\section{Concluding Remarks and Open Problems}\label{s:concrem}

As observed by Thomas Bloom (personal communication), the assumption
$\prt_S(A)\le(1-\gam)|A||S|$ can be equivalently written as
$\<1_A\circ1_A,1_S\>\ge\gam|A||S|$; here $1_A$ and $1_S$ are the indicator
functions of $A$ and $S$, respectively, the angle brackets are used for the
scalar product of complex-valued functions on the underlying group $G$, and
the ``skew convolution'' $f\circ g$ of the functions $f$ and $g$ is defined
by $(f\circ g)(z)=\sum_{x\in G}f(x)g(x+z)$. One can use H\"older's inequality
and basic Fourier analysis to conclude that if $S$ is dissociated, then
$\prt_S(A)\le(1-\gam)|A||S|$ implies $|A|>\exp(c\gam^2|S|)$, with an absolute
constant $c>0$. This conclusion is weaker than that given by
Theorem~\reft{generalcase}, but it requires dissociativity only (instead of
the more restrictive independence).

Our principal results, Theorems~\reft{exp234} and~\reft{generalcase}, show
that a subset of an abelian group with small edge boundary must be large.
It would be very interesting to understand exactly \emph{why} this happens.
What can be said about the \emph{structure} of a set with small edge
boundary?

In connection with Theorem~\reft{generalcase}, it is interesting to determine
whether the following is true: for any generating subset $S$ of a finite
abelian group $G$, if $\prt_S(A)\le(1-\gam)n|A|$ with $n:=\rk G$ and real
$\gam\in(0,1]$, then $|A|\ge4^{(1-1/d)\gam n}$ (where $d$ is the minimal
order of an element of $S$). It would also be interesting to find out whether the
coefficient $1-1/d$ in the exponent can be improved, or dropped altogether,
in the case where $G$ is homocyclic of exponent $\exp(G)\ge 5$ (the case
$\exp(G)\le 4$ is settled by Theorem~\reft{exp234} and Example~\refx{1}).

In the course of the proof of Theorem~\reft{avweightZ}, we have shown that
any finite, non-empty downset $A\seq\Z^n_{\ge 0}$ satisfies
\begin{equation}\label{e:LW+}
  n|A| \le |\pi_1(A)|\longp|\pi_n(A)| + \frac12 |A| \log_2|A|,
\end{equation}
where $\pi_1\longc\pi_n$ are projections onto the coordinate hyperplanes. In
fact, since compression can only reduce the sizes of the projections
$\pi_i(A)$, inequality \refe{LW+} holds true for \emph{any} finite, non-empty
set $A\seq\Z^n$, not necessarily a downset. Interestingly, this inequality
does not follow from the famous Loomis-Whitney inequality
$|\pi_1(A)|\dotsb|\pi_n(A)|\ge|A|^{n-1}$: for instance, the latter does not
exclude the possibility that there is a set $A\seq\Z^3$ with $|A|=5$ and
$|\pi_1(A)|=|\pi_2(A)|=|\pi_3(A)|=3$, while ~\refe{LW+} shows that such a set
cannot exist.

For Theorem~\reft{repa}, it would be interesting to determine the best
possible coefficient for homocyclic groups of exponent larger than $3$.

\section*{Acknowledgments} 
We are grateful to Thomas Bloom for the observation at the beginning of
Section~\refs{concrem}.

\bibliographystyle{amsplain}


\begin{dajauthors}
\begin{authorinfo}[lev]
  Vsevolod F. Lev\\
  Department of Mathematics\\
  The University of Haifa at Oranim\\
  Tivon 36006\\
  Israel\\
  seva@math.haifa.ac.il\\
  \end{authorinfo}
\end{dajauthors}

\end{document}